\documentclass[12pt]{amsart}

\usepackage{amssymb,amsmath,amsthm,enumerate}

\newcommand{\excise}[1]{}

\newcommand{\bull}{_{\bullet}}
\newcommand{\hh}[3]{h^{{#1}} \big( {#2} , {#3} \big) }
\newcommand{\HH}[3]{H^{{#1}} \big( {#2} , {#3} \big) }
\newcommand{\KK}{\mathbf{K}}
\newcommand{\lra}{\longrightarrow}
\newcommand{\NN}{\mathbb{N}}
\newcommand{\OO}{\mathcal{O}}
\newcommand{\QQ}{\mathbb{Q}}
\newcommand{\RR}{\mathbb{R}}
\newcommand{\vecbull}{_{\Vec{\bullet}}}

\DeclareMathOperator{\Image}{Im}
\DeclareMathOperator{\interior}{int}
\DeclareMathOperator{\supp}{supp}
\DeclareMathOperator{\vol}{vol}

\theoremstyle{plain}
\newtheorem{theorem}{Theorem}

\theoremstyle{definition}
\newtheorem{definition}[theorem]{Definition}
\newtheorem{remark}[theorem]{Remark}

\begin{document}%%%%%%%%%%%%%%%%%%%%%%%%%%%%%%%%%%%%%%%%%%%%%%%%%%%%%%
%%%%%%%%%%%%%%%%%%%%%%%%%%%%%%%%%%%%%%%%%%%%%%%%%%%%%%%%%%%%%%%%%%%%%%

\mbox{}
\vspace{-1.1ex}
\title{Multigraded Fujita Approximation}
\author{Shin-Yao Jow}
\address{Department of Mathematics,
University of Pennsylvania,
Philadelphia, PA 19104}
\email{\texttt{jows@math.upenn.edu}}
%\thanks{}
\date{11 May 2010}

\begin{abstract}
The original Fujita approximation theorem states that the volume of a big divisor $D$ on a projective variety $X$ can always be approximated arbitrarily closely by the self-intersection number of an ample divisor on a birational modification of $X$. One can also formulate it in terms of graded linear series as follows: let $W_{\bullet} = \{ W_k \}$ be the complete graded linear series associated to a big divisor $D$: \[
  W_k = H^0\big(X,\mathcal{O}_X(kD)\big).  \]
For each fixed positive integer $p$, define $W^{(p)}_{\bullet}$ to be the graded linear subseries of $W_{\bullet}$ generated by $W_p$: \[
 W^{(p)}_{m}=\begin{cases}
                 0, &\text{if $p\nmid m$;}\\
                 \mathrm{Image} \big( S^k W_p \rightarrow W_{kp} \big), &\text{if $m=kp$.}
                \end{cases}  \]  
Then the volume of $W^{(p)}_{\bullet}$ approaches the volume of $W_{\bullet}$ as $p\to\infty$. We will show that, under this formulation, the Fujita approximation theorem can be generalized to the case of multigraded linear series.
\end{abstract}

\keywords{Fujita approximation, multigraded linear series, Okounkov body}
\subjclass[2000]{14C20}

\maketitle

%%%%%%%%%%%%%%%%%%%%%%%%%%%%%%%%%%%%%%%%%%%%%%%%%%%%%%%%%%%%%%%%%%%%%%
\section{Introduction}%%%%%%%%%%%%%%%%%%%%%%%%%%%%%%%%%%%%%%%%%%%%%%%
%%%%%%%%%%%%%%%%%%%%%%%%%%%%%%%%%%%%%%%%%%%%%%%%%%%%%%%%%%%%%%%%%%%%%%

Let $X$ be an irreducible variety of dimension $d$ over an algebraically closed field $\KK$, and let $D$ be a (Cartier) divisor on $X$. When $X$ is projective, the following limit, which measures how fast the dimension of the section space $\HH{0}{X}{\OO_X(mD)}$ grows, is called the \emph{volume} of $D$: \[
\vol(D)=\vol_X(D)=\lim_{m\to\infty} \frac{\hh{0}{X}{\OO_X(mD)}}{m^d/d!}. \]
One says that $D$ is \emph{big} if $\vol(D)>0$. It turns out that the volume is an interesting numerical invariant of a big divisor (\cite[\S 2.2.C]{Laz}), and it plays a key role in several recent works in birational geometry (\cite{BDPP}, \cite{Tsu}, \cite{HM}, \cite{Tak}).

When $D$ is ample, one can show that $\vol(D)=D^d$, the self-intersection number of $D$. This is no longer true for a general big divisor $D$, since $D^d$ may even be negative. However, it was shown by Fujita \cite{Fuj} that the volume of a big divisor can always be approximated arbitrarily closely by the self-intersection number of an ample divisor on a birational modification of $X$. This theorem, known as \emph{Fujita approximation}, has several implications on the properties of volumes, and is also a crucial ingredient in \cite{BDPP} (see \cite[\S 11.4]{Laz} for more details).

In their recent paper \cite{LM}, Lazarsfeld and Musta\c{t}\v{a} obtained, among other things, a generalization of Fujita approximation to \emph{graded linear series}. Recall that a graded linear series $W\bull = \{ W_k \}$ on a (not necessarily projective) variety $X$ associated to a divisor $D$ consists of finite dimensional vector subspaces
\[  W_k \ \subseteq \ \HH{0}{X}{\OO_X(kD)}  \]
for each $k \ge 0$, with $W_0 = \KK$, such that \[
W_k \cdot W_\ell \ \subseteq \ W_{k + \ell}   \] 
for all $k , \ell \ge 0 $.  Here the product on the left denotes the image of $W_k \otimes W_\ell$ under the multiplication map $\HH{0}{X}{\OO_X(kD)} \otimes \HH{0}{X}{\OO_X(\ell D)} \lra \HH{0}{X}{\OO_X(({k + \ell})D)}$. In order to state the Fujita approximation for $W\bull$, they defined, for each fixed positive integer $p$, a graded linear series $W^{(p)}\bull$ which is the sub graded linear series of $W\bull$ generated by $W_p$: \[
 W^{(p)}_{m}=\begin{cases}
                 0, &\text{if $p\nmid m$;}\\
                 \Image \big( S^k W_p \lra W_{kp} \big), &\text{if $m=kp$.}
                \end{cases}  \]  
Then under mild hypotheses, they showed that the volume of $W^{(p)}\bull$ approaches the volume of $W\bull$ as $p\to\infty$. See \cite[Theorem~3.5]{LM} for the precise statement, as well as \cite[Remark~3.4]{LM} for how this is equivalent to the original statement of Fujita when $X$ is projective and $W\bull$ is the complete graded linear series associated to a big divisor $D$ (i.e. $W_k=\HH{0}{X}{\OO_X(kD)}$ for all $k\ge 0$).

The goal of this note is to generalize the Fujita approximation theorem to \emph{multigraded linear series}. We will adopt the following notations from \cite[\S 4.3]{LM}: let $D_1, \ldots, D_r$ be divisors on $X$. For $\Vec{m} = (m_1,\ldots, m_r) \in \NN^r$ we write $\Vec{m}D = \sum m_i D_i$, and we put 
$|\Vec{m}| = \sum |m_i|$. 
\begin{definition}
A \emph{multigraded linear series} $W\vecbull$ on $X$ associated to the $D_i$'s consists of finite-dimensional vector subspaces
\[  W_{\Vec{k}} \ \subseteq \ \HH{0}{X}{\OO_X(\Vec{k}D)}\]
for each $\Vec{k} \in \NN^r$, with $W_{\Vec{0}} = \KK$, such that
\[  W_{\Vec{k}} \cdot W_{\Vec{m}} \ \subseteq \ W_{\Vec{k} + \Vec{m}}, \]
where the multiplication on the left denotes the image of $W_{\Vec{k}} \otimes W_{\Vec{m}} $ under the natural map $\HH{0}{X}{\OO_X(\Vec{k}D) } \otimes \HH{0}{X}{\OO_X(\Vec{m}D) } \lra \HH{0}{X}{\OO_X((\Vec{k}+\Vec{m})D) } $.
\end{definition}

Given $\Vec{a} \in \NN^r$, denote by $W_{\Vec{a},\bullet}$ the singly graded linear series  associated to the divisor $\Vec{a}D$ given by the subspaces $W_{k \Vec{a}} \subseteq \HH{0}{X}{\OO_X(k\Vec{a}D)}$. Then put
\[ \vol_{W\vecbull}(\Vec{a}) \ = \ \vol(W_{\Vec{a},\bullet}) \]
(assuming that this quantity is finite). It will also be convenient for us to consider $W_{\Vec{a},\bullet}$ when $\Vec{a}\in\QQ^r_{\ge 0}$, given by \[
 W_{\Vec{a},k}=\begin{cases}
                W_{k \Vec{a}}, &\text{if $k \Vec{a}\in \NN^r$;}\\
                0, &\text{otherwise.}
               \end{cases}   \]
               
Our multigraded Fujita approximation, similar to the singly-graded version, is going to state that (under suitable conditions) the volume of $W\vecbull$ can be approximated by the volume of the following finitely generated sub multigraded linear series of $W\vecbull$:

\begin{definition}
 Given a multigraded linear series $W\vecbull$ and a positive integer $p$, define $W\vecbull^{(p)}$ to be the sub multigraded linear series of $W\vecbull$ generated by all $W_{\Vec{m}_i}$ with $|\Vec{m}_i|=p$, or concretely \[
   W^{(p)}_{\Vec{m}}=\begin{cases}
     \displaystyle    \qquad  0, &\text{if $p\nmid |\Vec{m}|$;} \\
   \displaystyle \sum_{\substack{|\Vec{m}_i|=p \\ \Vec{m}_1+\cdots+\Vec{m}_k=\Vec{m}}} W_{\Vec{m}_1}\cdots W_{\Vec{m}_k}, &\text{if $|\Vec{m}|=kp$.}
                        \end{cases}  \]
\end{definition}

We now state our multigraded Fujita approximation when $W\vecbull$ is a complete multigraded linear series, since this is the case of most interest and allows for a more streamlined statement. We will point out in Remark~\ref{r:assumptions} afterward what assumptions on $W\vecbull$ are actually needed in the proof.

\begin{theorem} \label{t:main}
 Let $X$ be an irreducible projective variety of dimension $d$, and let $D_1,\ldots,D_r$ be big divisors on $X$. Let $W\vecbull$ be the complete multigraded linear series associated to the $D_i$'s, namely \[
 W_{\Vec{m}}=\HH{0}{X}{\OO_X(\Vec{m}D)} \]
for each $\Vec{m}\in \NN^r$. Then given any $\varepsilon >0$, there exists an integer $p_0=p_0(\varepsilon)$ having the property that if $p\ge p_0$, then 
\begin{equation} \label{eq:1}
 \bigg| 1- \frac{\vol_{W^{(p)}\vecbull}(\Vec{a})}{\vol_{W\vecbull}(\Vec{a})} \bigg|< \varepsilon 
\end{equation}
 for all $\Vec{a}\in \NN^r$.
\end{theorem}

\noindent\textbf{Acknowledgments.} The author would like to thank Robert Lazarsfeld for raising this question during an email correspondence.

%%%%%%%%%%%%%%%%%%%%%%%%%%%%%%%%%%%%%%%%%%%%%%%%%%%%%%%%%%%%%%%%%%%%%%
\section{Proof of Theorem~\ref{t:main}}
%%%%%%%%%%%%%%%%%%%%%%%%%%%%%%%%%%%%%%%%%%%%%%%%%%%%%%%%%%%%%%%%%%%%%%

The main tool in our proof is the theory of \emph{Okounkov bodies} developed systematically in \cite{LM}. Given a graded linear series $W\bull$ on a $d$-dimensional variety $X$, its Okounkov body $\Delta(W\bull)$ is a convex body in $\RR^d$ that encodes many asymptotic invariants of $W\bull$, the most prominent one being the volume of $W\bull$, which is precisely $d!$ times the Euclidean volume of $\Delta(W\bull)$. The idea first appeared in Okounkov's papers \cite{Oko96} and \cite{Oko03} in the case of complete linear series of ample line bundles on a projective variety. Later it was further developed and applied to much more general graded linear series by Lazarsfeld-Musta\c{t}\v{a} \cite{LM}, and also independently by Kaveh-Khovanskii \cite{KK08,KK09}.

\begin{proof}[Proof of Theorem~\ref{t:main}]
 Let $T=\{ (a_1,\ldots,a_r)\in \RR_{\ge 0}^r \mid a_1+\cdots+a_r=1 \}$, and let $T_{\QQ}$ be the set of all points in $T$ with rational coordinates. The fraction inside \eqref{eq:1} is invariant under scaling of $\Vec{a}$ due to homogeneity, hence it is enough to prove \eqref{eq:1} for $\Vec{a}\in T_{\QQ}$.
 
 Let $\Delta(W\vecbull)\subseteq \RR^d\times\RR^r$ be the global Okounkov cone of $W\vecbull$ as in \cite[Theorem~4.19]{LM}, and let $\pi\colon \Delta(W\vecbull)\to \RR^r$ 
be the projection map. For each $\Vec{a}\in T$ we write $\Delta(W\vecbull)_{\Vec{a}}$ for the fiber $\pi^{-1}(\Vec{a})$.  We also define in a similar fashion the convex cone $\Delta(W^{(p)}\vecbull)$ and the convex bodies $\Delta(W^{(p)}\vecbull)_{\Vec{a}}$. By \cite[Theorem~4.19]{LM}, 
\begin{equation} \label{eq:2}
  \Delta(W\vecbull)_{\Vec{a}}=\Delta(W_{{\Vec{a}},\bullet}) \quad 
                                         \text{for all $\Vec{a}\in T_{\QQ}$.} 
\end{equation}
Note that although \cite[Theorem~4.19]{LM} requires $\Vec{a}$ to be in the relative interior of $T$, here we know that \eqref{eq:2} holds even for those $\Vec{a}$ in the boundary of $T$ because the big cone of $X$ is open and $W\vecbull$ was assumed to be the complete multigraded linear series. By the singly-graded Fujita approximation, $\vol(W_{{\Vec{a}},\bullet})$ can be approximated arbitrarily closely by $\vol(W^{(p)}_{{\Vec{a}},\bullet})$ if $p$ is sufficiently large. (Here by $W^{(p)}_{{\Vec{a}},\bullet}$ we mean $W^{(p)}\vecbull$ restricted to the $\Vec{a}$ direction, which certainly contains $(W_{{\Vec{a}},\bullet})^{(p)}$.) Hence given any finite subset $S\subset T_{\QQ}$ and any $\varepsilon'>0$, we have \[
 \vol\bigl(\Delta(W^{(p)}\vecbull)_{\Vec{a}}\bigr)\ge 
                 \vol\bigl(\Delta(W\vecbull)_{\Vec{a}}\bigr)-\varepsilon' \quad 
                                 \text{for all $\Vec{a}\in S$} \]
as soon as $p$ is sufficiently large.

Because the function $\Vec{a}\mapsto \vol\bigl(\Delta(W\vecbull)_{\Vec{a}}\bigr)$ is uniformly continuous on $T$, given any $\varepsilon'>0$, we can partition $T$ into a union of polytopes with disjoint interiors $T=\bigcup T_i$, in such a way that the vertices of each $T_i$ all have rational coordinates, and on each $T_i$ we have a constant $M_i$ such that 
\begin{equation} \label{eq:3}
  M_i \le \vol\bigl(\Delta(W\vecbull)_{\Vec{a}}\bigr) \le M_i + \varepsilon' \quad
                  \text{for all $\Vec{a}\in T_i$.} 
\end{equation} 
Let $S$ be the set of vertices of all the $T_i$'s. Then as we saw in the end of the previous paragraph, as soon as $p$ is sufficiently large we have 
\begin{equation} \label{eq:4}
 \vol\bigl(\Delta(W^{(p)}\vecbull)_{\Vec{a}}\bigr)\ge 
                 \vol\bigl(\Delta(W\vecbull)_{\Vec{a}}\bigr)-\varepsilon' \quad
                      \text{for all $\Vec{a}\in S$.}
\end{equation}
We claim that this implies 
\begin{equation} \label{eq:5}
 \vol\bigl(\Delta(W^{(p)}\vecbull)_{\Vec{a}}\bigr)\ge 
                 \vol\bigl(\Delta(W\vecbull)_{\Vec{a}}\bigr)-2\varepsilon' \quad
                                \text{for all $\Vec{a}\in T_{\QQ}$.}
\end{equation}
To show this, it suffices to verify it on each of the $T_i$'s. Let $\Vec{v}_1,\ldots,\Vec{v}_k$ be the vertices of $T_i$. Then each $\Vec{a}\in T_i$ can be written as a convex combination of the vertices: $\Vec{a}=\sum t_j \Vec{v}_j$ where each $t_j\ge 0$ and $\sum t_j=1$. Since $\Delta(W^{(p)}\vecbull)$ is convex, we have \[
 \Delta(W^{(p)}\vecbull)_{\Vec{a}}\ \supseteq\ \sum t_j\, \Delta(W^{(p)}\vecbull)_{\Vec{v}_j}, \]
where the sum on the right means the Minkowski sum. By \eqref{eq:3} and \eqref{eq:4}, the volume of each $\Delta(W^{(p)}\vecbull)_{\Vec{v}_j}$ is at least $M_i-\varepsilon'$, hence by the Brunn-Minkowski inequality \cite[Theorem~5.4]{KK08}, we have \[
 \vol\bigl(\Delta(W^{(p)}\vecbull)_{\Vec{a}}\bigr)\ge M_i-\varepsilon' \quad
                             \text{for all $\Vec{a}\in T_i\cap T_{\QQ}$.} \]
This combined with \eqref{eq:3} shows that \eqref{eq:5} is true on $T_i\cap T_{\QQ}$, hence it is true on $T_{\QQ}$ since the $T_i$'s cover $T$. 

Since \eqref{eq:1} follows from \eqref{eq:5} by choosing a suitable $\varepsilon'$, the proof is thus complete.
\end{proof}

\begin{remark} \label{r:assumptions}
 In the statement of Theorem~\ref{t:main} we assume that $W\vecbull$ is the complete multigraded linear series associated to big divisors. But in fact since the main tool we used in the proof is the theory of Okounkov bodies established in \cite{LM}, in particular \cite[Theorem~4.19]{LM}, the really indispensable assumptions on $W\vecbull$ are the same as those in \cite{LM} (which they called Conditions (A$'$) and (B$'$), or (C$'$)). The only place in the proof where we invoke that we are working with a complete multigraded linear series is the sentence right after \eqref{eq:2}, where we want to say that \eqref{eq:2} holds not only in the relative interior of $T$ but also in its boundary. Hence if $W\vecbull$ is only assumed to satisfy Conditions (A$'$) and (B$'$), or (C$'$), then given any $\varepsilon>0$ and any compact set $C$ contained in 
 $T\cap \interior\bigl(\supp(W\vecbull)\bigr)$, there exists an integer 
 $p_0=p_0(C,\varepsilon)$ such that if $p\ge p_0$ then \[
   \vol_{W^{(p)}\vecbull}(\Vec{a})> \vol_{W\vecbull}(\Vec{a})-\varepsilon \]
for all $\Vec{a}\in C\cap T_{\QQ}$.
\end{remark}

%%%%%%%%%%%%%%%%%%%%%%%%%%%%%%%%%%%%%%%%%%%%%%%%%%%%%%%%%%%%%%%%%%%%%%
%%%%%%%%%%%%%%%%%%%%%%%%%%%%%%%%%%%%%%%%%%%%%%%%%
%%%%%%%%%%%%%%%%%%%%%%%%%%%%%%%%%%%%%%%%%%%%%%%%%%%%%%%%%%%%%%%%%%%%%%

%%%%%%%%%%%%%%%%%%%%%%%%%%%%%%%%%%%%%%%%%%%%%%%%%%%%%%%%%%%%%%%%%%%%%%
%%%%%%%%%%%%%%%%%%%%%%%%%%%%%%%%%%%%%%%%%%%%%%%%%%%%%%%%%%%%%%%%%%%%%%
\end{document}